\documentclass[twoside]{article}

\usepackage{amsmath,amsthm,amssymb,graphicx,upquote,verbatim}
\usepackage{mathptmx}
\usepackage[text={12.5cm,19cm},centering,paperwidth=17cm,paperheight=24cm]{geometry}
\usepackage[fontsize=10.4pt]{scrextend}
\def\titlerunning#1{\gdef\titrun{#1}}
\makeatletter
\def\author#1{\gdef\autrun{\def\and{\unskip, }#1}\gdef\@author{#1}}
\def\address#1{{\def\and{\\\hspace*{15.6pt}}\renewcommand{\thefootnote}{}\footnote{#1}}\markboth{\autrun}{\titrun}}
\makeatother
\def\email#1{email: \href{mailto:#1}{#1} }
\def\subjclass#1{\par\bigskip\noindent\textbf{Mathematics Subject Classification 2020.} #1}
\def\keywords#1{\par\smallskip\noindent\textbf{Keywords.} #1}

\newenvironment{acknowledgments}{\bigskip\small\noindent\textit{Acknowledgments.}}{\par}
\newtheorem{thm}{Theorem}[section]

\theoremstyle{definition}

\numberwithin{equation}{section}
\def\Im{\hbox{\rm Im\kern .8pt }}
\def\Re{\hbox{\rm Re\kern .8pt }}
\def\R{{\mathbb R}}
\def\C{{\mathbb C}}
\def\Cp{{\C_+}}\def\Cm{{\C_-}}
\def\D{{\Delta}}\def\Dm{{\Delta_-}}
\def\Rrp{\Re r_+}
\def\pO{\partial\Omega}
\def\pput(#1,#2)#3{\noindent\smash{\raise#2pt\hbox to 0pt {\kern #1pt #3\hss}}\ignorespaces}

\usepackage[hyperfootnotes=false,colorlinks=true,allcolors=blue]{hyperref}

\begin{document}

\titlerunning{AAA-least squares rational approximation}
\title{\textbf{AAA-least squares rational approximation and solution of Laplace problems}}
\author{Stefano Costa \and Lloyd N. Trefethen}
\date{}

\maketitle


\address{S. Costa: IEEE Member, Piacenza, Italy; \email{stefano.costa@ieee.org} \and 
L. N. Trefethen: Mathematical Institute, University of
Oxford, Oxford OX4 4DY, UK; \email{trefethen@maths.ox.ac.uk}}

\begin{abstract}
A two-step method for solving planar Laplace problems
via rational approximation is introduced.  First complex rational approximations
to the boundary data are determined by AAA approximation, either globally
or locally near each
corner or other singularity.  The poles of these approximations outside the
problem domain are then collected and used for a global least-squares fit
to the solution.  Typical problems are solved in a second of laptop time to
8-digit accuracy, all the way up to the corners, and the conjugate harmonic
function is also provided.  The AAA-least squares
combination also offers a new method for avoiding spurious poles in
other rational approximation problems, and for greatly speeding them up in cases
with many singularities.  As a special case, AAA-LS approximation
leads to a powerful method for computing
the Hilbert transform or Dirichlet-to-Neumann map.
\end{abstract}

\section{\label{secintro}Introduction}
The aim of this paper is to introduce a new method for the
numerical solution of planar Laplace problems, based on
a combination of local complex rational approximations by
the AAA algorithm followed by a real linear least-squares
problem.  This method is an outgrowth of three previous
works~\cite{costa,lightning,AAA}, which we now briefly summarize.

The AAA algorithm (adaptive Antoulas--Anderson, pronounced
``triple-A'') is a fast and flexible method for
near-best complex rational approximation~\cite{AAA}.
Given a vector $Z$ of real or complex
sample points and a corresponding vector $F$ of data values, it finds
a rational function 
$r$ of specified degree or accuracy such that
\begin{equation}
r(Z)\approx F.
\label{approxprob}
\end{equation}
This is done by developing a barycentric representation for $r$
by alternating a nonlinear step of greedy selection of the
next barycentric support point with a linear least-squares
approximation step to determine barycentric weights.  If $F$
is obtained by a sampling a function $f(z)$ with singularities
at certain points of $Z$, such as logarithms and fractional
powers, then root-exponential convergence with respect to the
degree $n$ is typically achieved (i.e., errors $O(\exp(-C\sqrt
n\kern 1pt))$ for some $C>0$), with poles of the approximants $f$
clustering exponentially near the singularities~\cite{clustering}.
The standard implementation of AAA approximation is the code {\tt
aaa} in Chebfun~\cite{chebfun}.

The lightning Laplace solver is a method for solving Laplace problems
\begin{equation}
\Delta u = 0 \hbox{ on $\Omega$}, \quad
u = h(z) \hbox{ on $\pO$}
\label{laplaceprob}
\end{equation}
on a simply connected domain $\Omega$ in the plane, which
we parametrize for convenience by the complex variable
$z$~\cite{lightning}.  It also computes an analytic function
$f(z)$ such that $u = \Re f$.  This method first fixes poles with
exponential clustering near each corner of $\Omega$ or other point
where a singularity is expected.  A real linear least-squares
problem is then solved to determine a rational function in
$\Omega$ with the prescribed poles, plus a polynomial term (i.e.,
poles at infinity), whose real part matches the boundary data as
closely as possible.  The method converges root-exponentially with
respect to the number of poles and generalizes to Neumann boundary
data, multiply connected domains, and the Stokes and Helmholtz
equations~\cite{stokes,helmholtz}.  The standard implementation
is the MATLAB code {\tt laplace} available at~\cite{lightningcode}.

Although the lightning Laplace solver is fast and effective,
one would really like to solve Laplace problems by a method
more like the AAA algorithm, which allows the set $Z$ to be
completely arbitrary and adapts to the singularities of the
solution automatically rather than relying on a priori estimates of
pole clustering.  Two challenges have held back the development
of a AAA method for Laplace problems.  First, no barycentric
representation is known for real parts of rational functions.
Second, even if such a formula were available, there would remain
the fundamental problem of achieving approximation in a region
$\Omega$ based on values on the boundary $\pO$.  A AAA-style
approximation does not distinguish interior from exterior and
includes no mechanism to restrict poles to the latter.

These considerations led to the third contribution that this
work builds upon, published on arXiv by the first author
in 2020~\cite{costa}.  The upper row of Figure~\ref{fig1}
illustrates the idea as applied to the ``NA Digest model
problem''~\cite{digest}, an L-shaped region with boundary data
$u(z) = (\Re z)^2$.  First, complex AAA is used to approximate the
real data on the boundary.  The resulting analytic function is
complex (though real on $\pO$, up to the approximation accuracy),
with poles both inside and outside $\Omega$.  Then the poles in
$\Omega$ are discarded, leaving a set of poles outside $\Omega$
that are often clustered effectively for rational approximation.
The Laplace problem is solved by computing such an approximation
by linear least-squares fitting on $\pO$.

\begin{figure}[t]
\begin{center}
\vskip 8pt
\includegraphics[scale=.83]{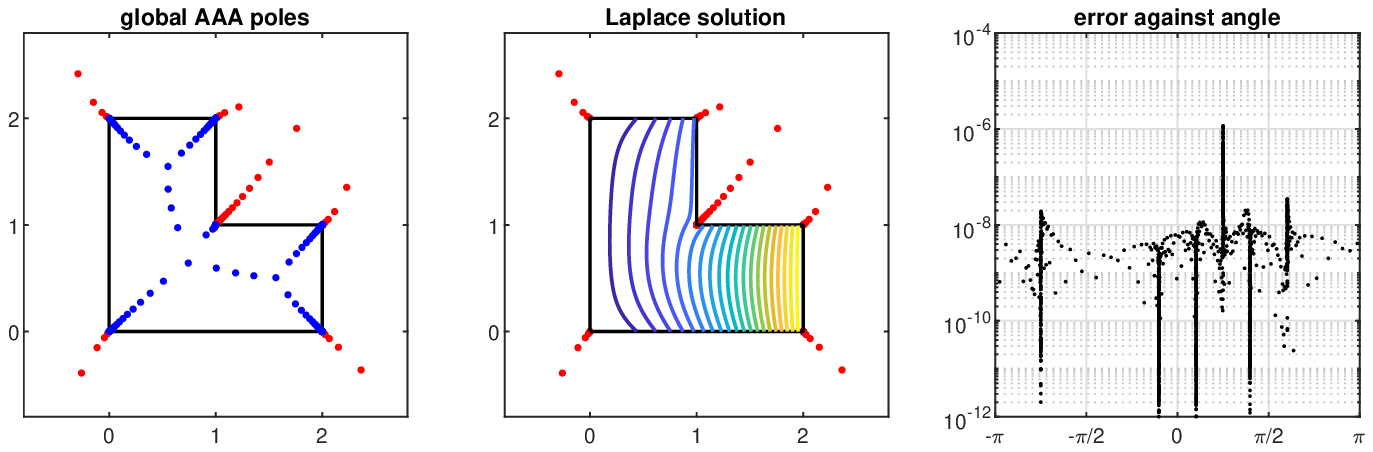}
\vskip 12pt
\includegraphics[scale=.83]{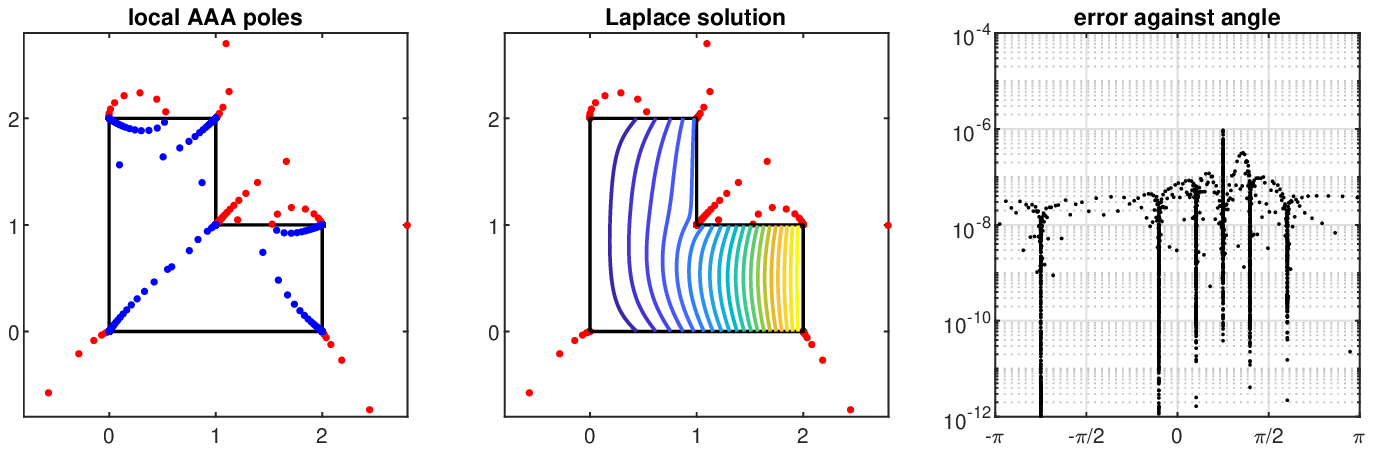}
\vspace{-16pt}
\end{center}
\caption{\small\label{fig1}Above, Costa's AAA-Laplace method from~{\rm\cite{costa}}.
A global AAA
approximation gives poles both inside and outside $\Omega$.  The poles inside
are discarded, and those outside are used for a linear least-squares fit.
Errors on the boundary in the
rightmost plot are plotted against angle with respect to the point $(1+i\kern .7pt )/2$.
This computation determines $u(0.99+0.99i)\approx 1.0267919261073$ to
10 digits of accuracy, but it takes 12 secs.\ of laptop time because
the AAA approximation has $294$ poles.
Below, the new local variant, in which the poles outside
$\Omega$ are determined by local AAA approximations
near each corner. 
The computation time falls to $0.67$ secs.\ because the
AAA problems are six times smaller, without much change
in accuracy.}
\end{figure}

In the form just described, the AAA-Laplace method can be quite
slow because of depending on AAA approximations with a large number
of poles.  In this article we propose a variation that often speeds
it up greatly, namely, to use local AAA approximations near each
singularity to choose the set of poles.  Since the cost of AAA
approximation grows with the fourth power of the number of poles,
this leads to a speedup potentially by a factor on the order of
the cube of the number of corners.  For the L-shaped example the
speedup is a factor of about $18$.

The AAA-Laplace method as presented in~\cite{costa} was actually
much slower than indicated in Figure~\ref{fig1} for an accidental
reason.  In that implementation, {\tt aaa} was invoked in its
default ``cleanup'' mode, which led to the removal of many poles
close to the singularities and a consequent need to compute AAA
approximations involving as many as $1000$ poles.  What that paper
interpreted as a halving of the number of digits of accuracy
due to discarding poles in $\Omega$ now seems to have been a
consequence of using the cleanup feature.  Throughout this paper,
we always call {\tt aaa} with ``cleanup off''.

\section{Laplace problems}
Our main interest is problems with corner singularities, since
this is where the power and convenience of rational functions are
most decisive.  However, the AAA approach can be effective for
smooth problems too.  Figure~\ref{randfig} presents an example.
An irregular domain $\Omega$ (bounded by a trigonometric
interpolant through 15 complex data points) is given with the
Laplace boundary condition $u(z)= -\log |z| $.  The vector $Z$
is constructed by sampling $\pO$ in $1000$ points, and a global
AAA fit to the boundary data with tolerance $10^{-8}$ yields $46$
poles in $\Omega$ and $30$ in $\C\backslash\overline{\Omega}$.
The interior poles are discarded, and a least-squares fit to the
boundary data is computed via a $1000\times 102$ matrix: 60 real
degrees of freedom for 30 poles and 42 for a polynomial term of
degree $20$.  The computation takes $0.7$ secs., and the maximum
error on $Z$ is $2.1\times 10^{-9}$.  A polynomial expansion
needs about 10 times as many degrees of freedom to achieve the
same accuracy, a ratio that would worsen exponentially for
more distorted regions according to the theory of the ``crowding
phenomenon'' in complex analysis~\cite[Thm.~5]{nm}.

\begin{figure}
\begin{center}
\vskip 8pt
\includegraphics[scale=.83]{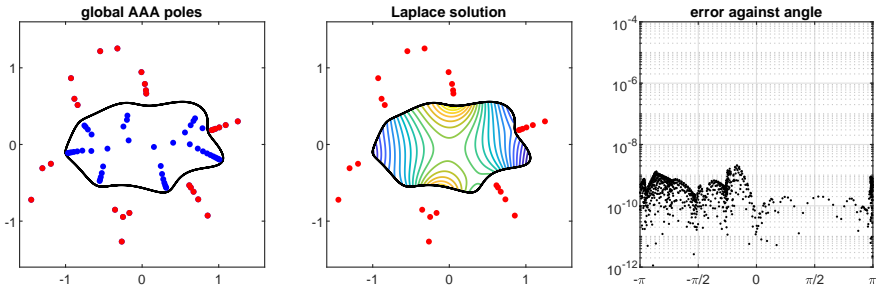}
\vspace{-16pt}
\end{center}
\caption{\small\label{randfig}A smooth Laplace problem solved by the global AAA-LS
method.  A global AAA approximation produces 46 poles inside $\Omega$ and
30 outside, and the latter are retained for a real least-squares problem that
also includes a polynomial term.
9-digit accuracy is achieved in $0.7$ secs.}
\end{figure}

We now turn to problems with singularities, typically at corners,
whose locations are assumed to be known in advance.   The local
variant of the AAA-LS algorithm proceeds in this manner:

{\em 1. Construct sample point vector $Z$ and fix corresponding
data values $H=h(Z)$.}

{\em 2. For each singularity, run AAA for nearby sample points
and data values.}

{\em 3. Discard poles in $\overline{\Omega}$ and retain poles
exterior to $\overline{\Omega}\kern .5pt$.}

{\em 4. Calculate real least-squares fit to boundary data,
including a polynomial term.}

{\em 5. Construct function handles for $u(z)$ and its analytic
extension $f(z)$.}

\noindent We give some mathematical and MATLAB details of each of these steps.
The global variant of the algorithm is the same except that step 2 involves
just a single global AAA approximant.
\smallskip

{\em 1. Construct sample point vector $Z$ and fix corresponding
data values $H=h(Z)$.} The problem domain $\Omega$ can be quite
arbitrary, and it can be multiply connected.  Typically $Z$
will consist of hundreds or thousands of points, which it is
simplest to specify in advance with exponential clustering
near singularities.  In MATLAB we use constructions like
\verb|logspace(-14,0,300)'| for a singularity at one endpoint of
$[\kern .5pt 0,1]$ and \verb|tanh(linspace(-16,16,600)')| for
singularities at both endpoints of $[-1,1]$.  If \hbox{AAA-LS}
software were to be developed analogous to the \verb|laplace| code
of~\cite{lightningcode} for the lightning Laplace method, then
it would be worthwhile placing sample points more strategically
to avoid having too many more rows in the matrix than necessary.

{\em 2. For each singularity, run AAA for nearby sample points
and data values.} We use the simplest choice: each point of $Z$
is associated with whichever singularity it is closest to (on the
same boundary component, if the geometry is multiply connected
so there are several boundary components).  The Chebfun command
{\tt aaa} is invoked with \verb|'cleanup','off'|, and throughout
this paper we specify a AAA tolerance of $10^{-8}$.

{\em 3. Discard poles in $\overline{\Omega}$ and retain poles exterior to
$\overline{\Omega}\kern .5pt$.}
The {\tt aaa} code returns highly accurate pole locations computed via
a matrix generalized eigenvalue problem described in~\cite{AAA}.
To distinguish those inside and outside $\Omega$, we use the complex variant 
\verb|inpolygonc = @(z,w) inpolygon(real(z),imag(z),real(w),imag(w))|
of the {\tt inpolygon} \kern .8pt command.

{\em 4. Calculate real least-squares fit to boundary data, including a polynomial term.}
If {\tt pol} is a row vector of the poles from 
step 3 and {\tt n} is a small nonnegative integer, the sequence
{\small
\begin{verbatim}
    d = min(abs(Z-pol),[],1);
    P = Z.^(0:n); Q = d./(Z-pol);
    A = [real(P) real(Q) -imag(P) -imag(Q)];
    c = reshape(A\H,[],2)*[1;1i];
\end{verbatim}
\par}
\noindent
computes a complex coefficient vector $c$ for the function $f$ in the space spanned
by the polynomials of degree $n$ and the given poles such that $u = \Re f$ is the
least-squares fit to the data $H$ in the sample points.
The vector $d$ contains the distances of the poles to $Z$ and is used to
scale the columns of $Q$ to have $\infty$-norm 1.\ \ For
$n$ much larger than 10, however, numerical stability requires that the
monomials of \verb|Z.^(0:n)| be replaced by orthogonalizations
computed by the Vandermonde with Arnoldi procedure of~\cite{VA}.
This can be done by replacing \verb|P = Z.^(0:n)| by 
\verb|[Hes,P] = VAorthog(Z,n)|,
where the code {\tt VAorthog} comes from~\cite{stokes}
and is listed in the appendix.

The description and code above apply for bounded, simply connected
domains with Dirichlet boundary conditions.  For problems with
Neumann boundary conditions on some sides, the corresponding rows
of $A$ are modified appropriately.  For exterior domains, $z$
is replaced by $(z-z_c)^{-1}$ for some point $z_c$ in the hole.
For multiply-connected domains, additional columns of the form
$\log|z-z_j|$ must be added where $\{z_j\}$ are a set of a fixed
points, one in each hole~\cite{axler,series}.  In addition, new
columns are added corresponding to polynomials in $1/(z-z_j)$
for each $j$.

{\em 5. Construct function handles for $u(z)$ and its analytic extension $f(z)$.}
For convenience in making plots and other applications, it is desirable
to have functions that can be applied to matrices as
well as vectors.  Following the commands above this can be achieved with
\par

\noindent{\small
\begin{verbatim}
    f = @(z) reshape([z(:).^(0:n) d./(z(:)-pol)]*c,size(z));
    u = @(z) real(f(z)); v = @(z) imag(f(z));
\end{verbatim}
\par}
\noindent
When {\tt VAorthog} is used, the first line is replaced by
{\small
\begin{verbatim}
    f = @(z) reshape([VAeval(z(:),Hes) d./(z(:)-pol)]*c,size(z));
\end{verbatim}
\par}

\begin{figure}[t]
\begin{center}
\vskip 8pt
\includegraphics[scale=.83]{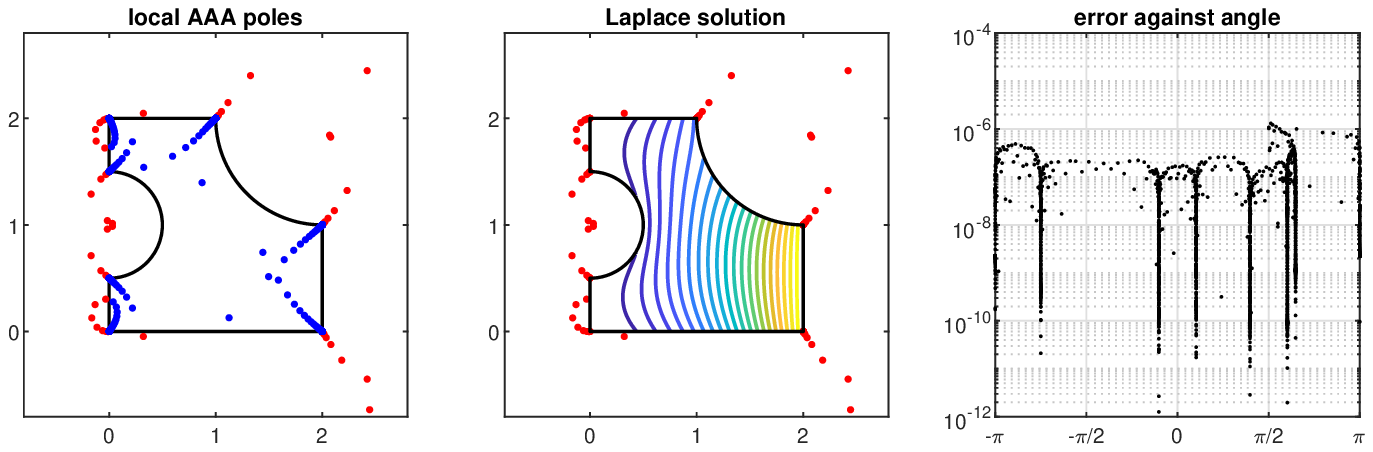}
\vskip 12pt
\includegraphics[scale=.83]{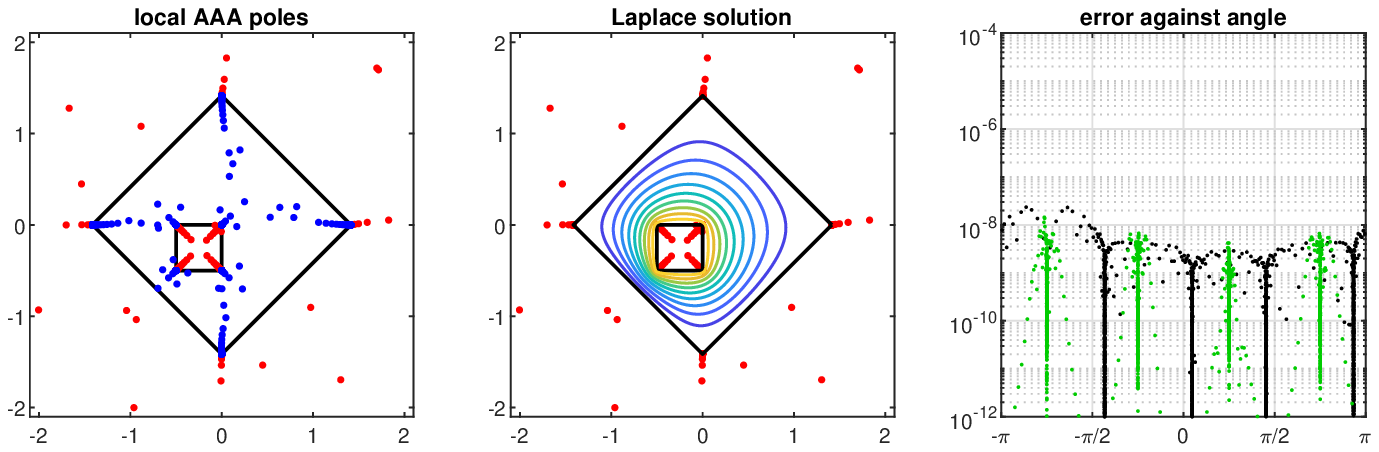}
\vspace{-16pt}
\end{center}
\caption{\small\label{circleL}Two examples of Laplace solutions by the 
local AAA-LS method.  Above, a square with two circular bites
removed.  The computation involves 102 poles outside the domain
and a polynomial of degree $20$.  Below, a multiply connected
domain, solved in 1.7 secs.\ with 397 poles outside the domain
and a polynomial of degree 40.  In the error plot, black dots
correspond to the outer boundary and green dots to the inner one.
The boundary data used for local-AAA pole location are not those
of the Laplace problem, as explained in the text.} \end{figure}

Figure~\ref{circleL} illustrates the method at work on two
examples.  In the first row, the L shape of Figure~\ref{fig1}
has been modified to a square with two circular bites removed.
No new issues arise here, as the method does not distinguish
between straight and curved sides, so long as they are smooth.
The second row shows a doubly-connected problem, and here some
new issues do arise.  First there is the use of polynomials
with respect to both $z$ and $(z-z_c)^{-1}$ as described above
(writing $z_c$ instead of $z_1$ since there is just one hole), as
well as the introduction of a $\log|z-z_c|$ term; we take $z_c =
-(1+i\kern .5pt )/4$.  The domain is discretized by 400 clustered
points on each of the eight side segments, and the polynomials
in $z$ and $(z-z_c)^{-1}$ are of degree 40.  A more fundamental
issue also arises in this problem.  The boundary data have been
taken as $1$ on the inner square and $0$ on the outer square,
a natural situation for a heat flow or electrostatics problem in
a doubly connected geometry.  Since these boundary conditions
are constant on each of the two boundary components, however,
the local AAA problems will be trivial and no poles at all will
be produced!  Clearly that is no route to an accurate solution,
so for this computation, poles have been generated by using an
artificial boundary condition (the square root of the product of
the distances to the eight corners) and then the least-squares
problem is solved with the boundary data actually prescribed.
The reader is justified if he/she finds this puzzling, and we
discuss the matter further in Section~\ref{theory}.

\begin{figure}
\begin{center}
\vskip 8pt
\includegraphics[scale=.75]{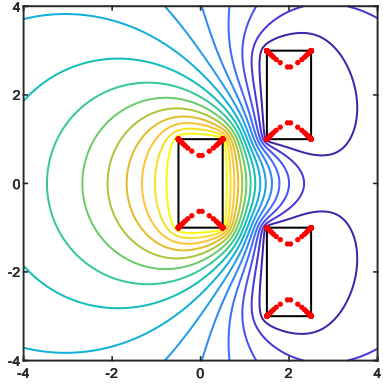}
\vspace{-16pt}
\end{center}
\caption{\small\label{three}Local AAA-LS
solution of a Laplace problem in an unbounded triply-connected domain,
requiring reciprocal polynomials with respect to three interior points
$c_j$ and also logarithm terms $\log |z-z_j|$.  The computation takes 2 secs.\ and
gives the value $u(1) \approx 0.64357510429036$ to 10-digit accuracy.}
\end{figure}

Our final example of this section is an unbounded region with three
rectangular holes, shown in Figure~\ref{three}.  The boundary
conditions are $u=1$ on the rectangle at the left and $u=0$ on
the other two, giving a natural interpretation as the potential
around three conductors.  Each boundary segment is discretized by
400 clustered points, so the least-squares matrix has 4800 rows.
The AAA fits lead to 52 poles inside a rectangle near each corner,
624 in total, and we also have a reciprocal-polynomial of degree
10 and one real logarithm term in each rectangle, bringing the
number of columns of the matrix to $2 \times (624+3\times 11) + 3
= 1317$.  A solution is computed in 2 secs.\ to 10-digit accuracy
as measured by the value at the point $z=1$ midway between the
rectangles, $u(1)\approx 0.64357510429036$.

A fine point to note in this triply-connected example is that the
point $z=\infty$ is a point of analyticity, in the interior of the
domain, so there should be no logarithmic term there, meaning that
the sum of the coefficients of the three log terms centered at
the points $z_1, z_2, z_3$ in the rectangles should be zero.  This
condition can be enforced by adding one more row to the matrix, or
(as was in fact done for the computation in the figure) by taking
the log columns of the matrix to correspond not to $\log|z-z_j|$
but to $\log|z-z_j| - \log|z-z_{j (\hbox{\scriptsize mod } 3)+1}|$.

\section{Conformal mapping}
A Laplace solver that also produces the harmonic conjugate of the solution,
hence its analytic extension, can
be used to compute conformal maps.  Details are given in~\cite{conf},
so here we give just one example of construction of the conformal map $g$ of
a simply-connected region $\Omega$ containing the point $z=0$ to the unit disk, with
$g(0) = 0$ and $g'(0)>0$.  The trick is to write $g$ in the form
\begin{equation}
g(z) = z\exp(f(z)), \quad f(z) = \log(g(z)/z),
\label{confmap}
\end{equation}
where $f$ is the unique nonzero analytic function
on $\Omega$ has real part $-\log |z|$
on $\pO$ and imaginary part $0$ at $z=0$.
Thus $f$ is obtained by solving a Laplace Dirichlet problem, and
(\ref{confmap}) then gives the map $g$.

\begin{figure}
\begin{center}
\vskip 8pt
\includegraphics[scale=.86]{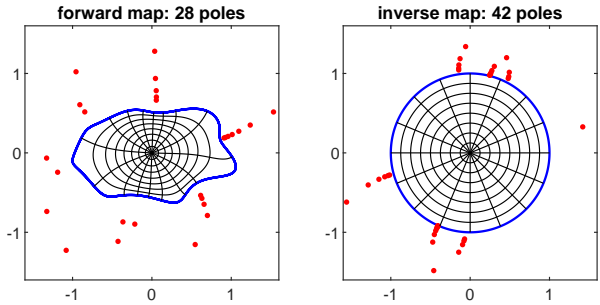}
\vspace{-16pt}
\end{center}
\caption{\small\label{confmapfig}Conformal map of the
region of Figure~\ref{randfig} by the global 
AAA-LS method.  The map is computed to 8-digit accuracy
in $0.8$ secs.\ and
the rational approximations in both directions are
evaluable in less than 1 $\mu$sec per point.  In the left image, the
poles differ slightly from those of Figure~\ref{randfig} because
a further AAA compression of $z\exp(f(z))$ has taken place.}
\end{figure}

Figure~\ref{confmapfig} illustrates this method for the smooth
region of Figure~\ref{randfig}, where $-\log|z|$ was already
the boundary condition.  Thus the conformal map comes from
exponentiating the analytic extension of the harmonic function
of Figure~\ref{randfig} and multiplying the result by $z$.
As described in~\cite{nm}, this result is then compressed by AAA
approximation, and another AAA approximation gives the inverse map.
See~\cite{conf} for extensions to multiply connected regions.

The speed of these computations is remarkable.  After an initial
0.9 secs.\ to construct the forward and inverse maps in this
example, they can be each then be evaluated in 0.3 $\mu$secs.\
per point.    For example, we take one million random points
uniformly distributed in the unit disk, map them conformally to
$\Omega$, then map these images back to the unit disk again.
The whole back-and-forth process takes $0.6$ secs., and the
maximum error in the million sample points is $1.1\times 10^{-8}$.

\section{Rational approximation without spurious poles}
Though the emphasis in this paper is on Laplace problems, AAA-LS
approximation also offers striking advantages for more general
rational approximations.  It may be much faster than AAA alone for
problems with a number of singularities, and since unwanted poles
can be discarded, it produces approximations guaranteed to have
desired properties of analyticity and stability.  Thus AAA-LS may
combat what Heather Wilber has called the ``spurious poles blues''
(discussed in~\cite{wilber}, though without this phrase).

\begin{figure}[t]
~~~~~\includegraphics[scale=.8]{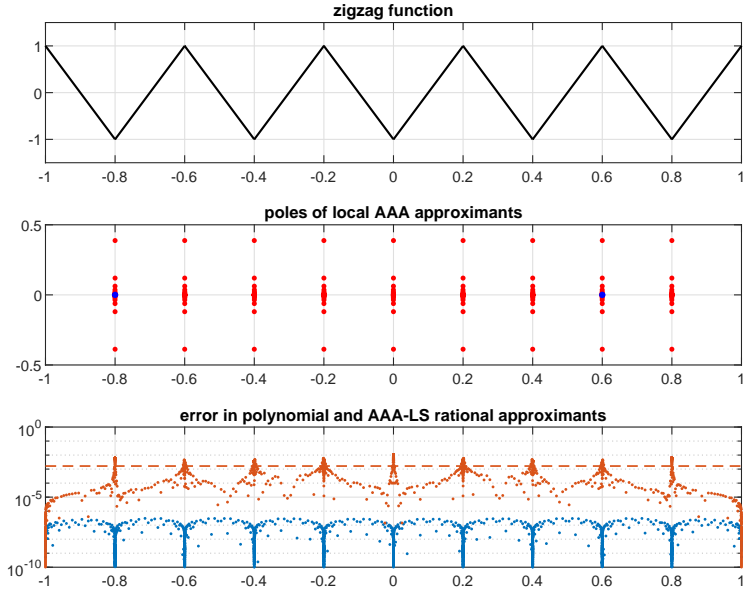}
\caption{\small\label{zigzag}Top, a real zigzag function on $[-1,1]$ to be approximated
over the whole interval by a single rational function.
Middle, the 466 poles determined by local AAA fits near each singularity, each of
degree 51 or 52.  Two poles lie in $[-1,1]$ and are discarded (blue).
Bottom, the resulting
errors in the AAA-LS fit show accuracy of $3\times 10^{-7}$.  
A polynomial with the 
same 962 degrees of freedom such as a Chebyshev interpolant (dots)
could have accuracy at best $10^{-3}$ (dashed line).}
\end{figure}

We illustrate both the speed and the robustness with an
example of approximating a real zigzag function on the
interval $[-1,1]$, shown in Figure~\ref{zigzag}.\ \ Knowing
that poles will need to cluster exponentially at the points
$-0.8, -0.6,\dots,0.8$, we set up a 3000-point grid consisting
of {\tt -0.9 + 0.2*tanh(linspace(-16,16,300))} and its nine
translates at centers $-0.7, -0.5, \dots, 0.9$.\ \ With straight
AAA approximation, poles in $[-1,1]$ virtually always appear.
They could be removed for input to a least-squares fit, but the
timing would still be very slow for the moderately large degrees
needed for effective approximation: $0.3$, $4.2$, and $35.3$
seconds on our laptop for degrees $50$, $200$, and $500$.\ \
By contrast, with its local AAA fits the AAA-LS method quickly
computes a good approximation.  In the figure, AAA-LS has been
run with AAA tolerance $10^{-8}$, leading to local fits each
of size $51$ and $52$ and hence quite speedy.  This gives $466$
poles all together, two of which lie in $[-1,1]$ and are discarded,
as shown in the middle panel of the figure.  A least-squares fit
with these 464 poles, plus a polynomial of degree 16, then gives
the error marked in blue in the bottom figure, with maximum error
$3.1\times 10^{-7}$.  The whole computation takes half a second,
and the resulting approximation can be evaluated in 5 $\mu$secs.\
per point.  By contrast a polynomial fit with the same 962 degrees
of freedom can have error no smaller than $1.6\times 10^{-3}$,
as marked by the red dashed line.  The red dots show the error
for a polynomial Chebyshev interpolant of that degree.

It appears that AAA-LS offers a flexible, fast, and reliable way to
compute near-best rational approximations with no unwanted poles.
Potential applications lie in many areas of computational science
and engineering.  An interesting question is, might AAA-LS be
further leveraged via a AAA-Lawson iteration as in~\cite{lawson}
to lead to truly minimax rational approximations in certain cases?\
\ For this to be possible, it would be necessary first to convert
the rational approximation to barycentric form.  We have not
explored this possibility.

\section{\label{hilbert}Computing the Hilbert transform}

If $u$ is a sufficiently smooth real function defined on the real line, its
Hilbert transform is the function $v$ defined by the principal value
integral
\begin{equation}
v(y) = {1\over \pi} \hbox{\kern 2pt PV\kern -2pt }
\int_{-\infty}^\infty {u(x)\over y-x} \kern .8pt dx.
\end{equation}
The transform can be interpreted as follows:
if $f$ is a complex analytic function in the upper half-plane with
$\Re f(x) = u(x)$ for $x\in\R$, then
$v(y) = \Im f(y)$.  Similar definitions and interpretations apply
to the unit circle and other contours.  Another name for the Hilbert
transform (essentially) is the Dirichlet-to-Neumann map.

It is evident that to compute the Hilbert transform numerically,
it suffices to find an analytic function in the upper half-plane
whose real part on $\R$ matches that of $u$ to sufficient accuracy.
The classical idea of this kind is to use a Fourier transform,
perhaps discretized on a finite interval by the Fast Fourier
Transform~\cite[p.~203]{henrici}.  For example, this is the method
used by the {\tt hilbert} command in the MATLAB Signal Processing
Toolbox.  But it is also possible to use rational approximations
instead of trigonometric polynomials, and numerical methods of
this kind have been proposed~\cite{mqmc,protasov,weideman}.

The AAA-LS method provides another natural approach based on
rational approximation, since poles in the upper half-plane can
be discarded to ensure the appropriate analyticity.  Indeed,
all of our AAA-LS Laplace solutions can be regarded as Hilbert
transforms, but on more general contours $\pO$.  A prototype code
for the real line can be written like this:

{\footnotesize
\begin{verbatim}
   function [v,f] = ht(u)
   X = logspace(-10,10,300)'; X = [X; -X];       % sampling grid
   [~,pol] = aaa(u(X),X,'cleanup',0);            % global AAA fit
   pol(imag(pol)>=0) = []; pol = pol.';          % discard unwanted poles
   d = min(abs(X-pol),[],1);                     % for column normalization
   A = d./(X-pol); A = [real(A) -imag(A)];       % fitting matrix
   c = reshape(A\u(X),[],2)*[1;1i];              % least-squares solve
   f = @(x) reshape((d./(x(:)-pol))*c,size(x));  % analytic extension
   v = @(x) imag(f(x));                          % Hilbert transform
\end{verbatim}
\par}
\noindent This is not an item of software---it is a proof
of concept.  Note that the sampling grid has been taken as 300
points exponentially spaced from $10^{-10}$ to $10^{10}$ and their
negatives, 600 points all together.   This would not be appropriate
for all functions, but it is a good starting point for a function
which loses analyticity possibly at $0$ and at $\infty$.  The code
does well at computing Hilbert transforms of the seven example
functions Weideman lists in Table~1 of his paper~\cite{weideman}.
In 0.6 secs.\ total on a laptop it produces results for these
seven example problems with relative accuracy in the range of
5--13 digits, as detailed in Table~\ref{thetable}.  We shall
not attempt systematic comparisons with other algorithms, but
as an indication of the nontriviality of these computations, we
mention that applying the MATLAB {\tt hilbert} command for $u(x) =
\exp(-x^2)$ on a grid of 1024 equispaced points in $[-20,20\kern
.5pt]$ gives an estimate of $v(2)$ with an error of $1.4\times
10^{-2}$, 11 orders of magnitude greater than the figure in
Table~\ref{thetable}.

\begin{table}[h]
\caption{\small\label{thetable}The example functions $u(x)$ from Table 1 by
Weideman~\cite{weideman} together with their Hilbert transforms $v(x)$
evaluated at the arbitrary point $x=2$.  In a total time of $0.6$ secs., the prototype
AAA-LS code {\tt ht} computes these numbers to 5--13 digits of relative
accuracy.}
\begin{center}
\def\ee{\hbox{\kern 1pt e}{-}}\def\em{\hphantom{-}}
\begin{tabular}{ccl}
Function $u$ & ~~Hilbert transform $v(2)$ & AAA-LS error \\ 
\hline \\[-6pt]
$1/(1+x^2)$       & ~~$0.400000000000000$ & ~~$-1.3\ee 12$ \\[2pt]
$1/(1+x^4)$       & ~~$0.415945165403851$ & ~~$-4.3\ee 14$ \\[2pt]
$\sin(x)/(1+x^2)$ & ~~$0.156805255543717$ & ~~$\em 3.4\ee 06$ \\[2pt]
$\sin(x)/(1+x^4)$ & ~~$0.121897775700258$ & ~~$-1.7\ee 07$ \\[2pt]
$\exp(-x^2)$      & ~~$0.340026217066066$ & ~~$\em 1.0\ee 13$ \\[2pt]
$\hbox{sech}(x)$  & ~~$0.506584586167368$ & ~~$\em 1.3\ee 10$ \\[2pt]
$\exp(-|x|)$      & ~~$0.328435745958114$ & ~~$-1.4\ee 12$ 
\end{tabular}
\end{center}
\end{table}

Figure~\ref{htplot} illustrates AAA-LS computation of the Hilbert transform
graphically for Weideman's final example,
\begin{equation}
u(x) = e^{-|x|}, \quad v(y) = \pi^{-1} \hbox{sign}(y) \bigl[ e^{|y|} E_1(|y|) +
e^{-|y|} Ei(|y|)\bigr],
\end{equation}
where $E_1$ and $Ei$ are the exponential integrals computed in
MATLAB by {\tt expint} and {\tt ei}.  For each of the values $L =
1,2,\dots, 6$, a sample grid of $60L$ points $y$ has been used
consisting of $30L$ points exponentially spaced from $10^{-L}$
to $10^{\kern 1pt L}$ and their negatives.  Rapid convergence is
observed to an accuracy of better than $10$ digits, despite
the singularity of $u$ at $x=0$.

\begin{figure}
\begin{center}
\vspace{8pt}
\includegraphics[scale=.83]{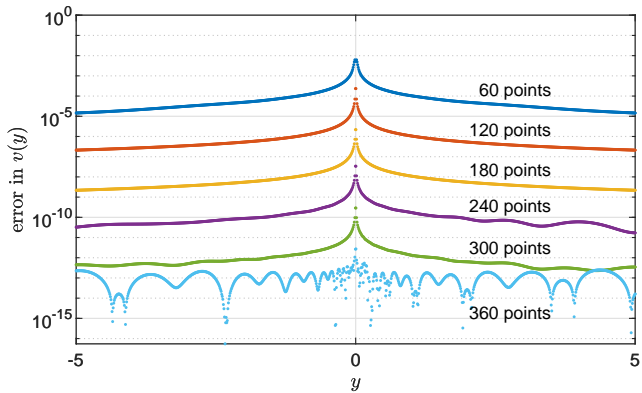}
\vspace{-12pt}
\end{center}
\caption{\label{htplot}Error at $1000$ points $y\in[-5,5]$
in the Hilbert transform
of $u(x) = \exp(-|x|)$ computed by the global AAA-LS method
from $60, 120, \dots , 360$ exponentially spaced samples.
This plot was produced in 2 secs.\ on a laptop.}
\end{figure}

The great flexibility of the AAA-LS method for computing the
Hilbert transform is to be noted.  It can work with arbitrary
data points, which need not be regularly spaced, and it delivers
a result as a global representation speedily evaluated via a
function handle.  No interpolation of data is required (see
discussion of this problem in~\cite{coscos}), and singularities
in $u(x)$ cause little degradation of accuracy so long as there
are sample points clustered nearby, as illustrated in the example
of Figure~\ref{htplot}.

Many generalizations of this AAA-LS Hilbert transform computation
are possible, including other contours both open and closed and
more general Riemann--Hilbert problems.

\section{\label{theory}Theoretical observations}

The core of the AAA-LS method (in its global form) is the following
idea, which we shall call the {\em pole symmetry principle}.
Suppose $r$ is a complex rational approximation that closely
approximates a real function $h$ on the boundary $\pO$ of a region
$\Omega$.  Then there is another complex rational function $r_+$,
{\em with poles only at the locations of the poles of $r$ outside
$\Omega$}, such that $\Re r_+$ also closely approximates $h$
on $\pO$.  The AAA-LS method finds $r$ by AAA approximation on
$\pO$, extracts its poles outside $\Omega$, and then finds $r_+$
by linear least-squares fitting on $\pO$.

In particular, for cases with singularities on $\pO$, rational
functions $r$ exist with root-exponential convergence to $h$ as
$n\to\infty$~\cite{lightning}.  Such approximations will usually
have poles that cluster exponentially on both sides of $\pO$
near each singularity.  The pole symmetry principle proposes
that we can discard all the poles inside $\Omega$, retaining
only the ones outside $\Omega$, and still get essentially the
same root-exponential convergence.

In this section we assess this idea.  Our conclusions can be
summarized as follows:

\begin{enumerate}
\item
If $\Omega$ is a half-plane or a disk, the pole symmetry principle holds exactly
(Theorems~\ref{thm1} and~\ref{thm2}).
\item
If $\Omega$ is a simply-connected domain with corners,
the pole symmetry principle fails in the worse case in that
$r_+$ may have no poles near $\pO$ even though they are needed
to resolve singularities; conversely it may have clusters
of poles near $\pO$ when they are not needed (examples shown
in Figure~\ref{lensfig}).\ \ However, both of these
situations are nongeneric.  For most problems, the principle
holds also on regions with corners.
\item 
If $\Omega$ is a simply-connected domain
bounded by an analytic curve, then in a certain
theoretical sense it can be
reduced to the case of a disk.  However, the
constants involved may be sufficiently adverse that in practice,
it may be more appropriate to think of $\Omega$ as 
a domain with corners.  Again the pole symmetry principle will usually hold
even if this cannot be guaranteed in the worst case.
\item
If $\Omega$ is a multiply-connected domain, then harmonic
functions in $\Omega$ can in general not be approximated by
rational functions: logarithmic terms are needed too.  Thus the
pole symmetry principle is inapplicable and a local rather than global
variant of AAA-LS should be used.
\end{enumerate}

To establish conclusion (1), let $\Cm$ and $\Cp$ denote the
open lower and upper complex half-planes, respectively, and let
$\|\cdot\|_E$ denote the supremum norm over a set $E$.  The two
assertions of the following theorem ensure that complex rational
approximation on $\R$ produces ``enough poles'' to solve the
Laplace problem on $\Cp$, and that it does not produce ``too many
poles'' to be efficient.

\begin{thm}
\label{thm1}
Given a bounded real continuous function $h$ on $\R$, let $u$ be the bounded
harmonic function in $\Cp$ with $u(x) =h(x)$ for $x\in\R$.  Suppose
there exists a rational
function $r$, also real on $\R$, such that $\|r-h\|_\R\le \varepsilon$
for some $\varepsilon \ge 0$.
Then there exists a rational function $r_+$ whose poles are
precisely the poles of\/ $r$ in $\Cm$ such that
$\|\Rrp - h\|_\R \le \varepsilon$ and thus by the maximum principle also
$\|\Rrp - u\|_\Cp\le \varepsilon$.  Conversely, if\/ $r_+$ is a rational function
analytic in $\Cp$ such that $\|\Rrp-u\|_\Cp\le \varepsilon$, then 
there exists a rational function $r$ whose poles are
the poles of\/ $r_+$ and their reflections in $\Cp$ such that
$\|r-h\|_\R\le \varepsilon$.  
\end{thm}

\begin{proof}
Given $r$ as indicated in the first assertion,
write $r(z) = (r_+(z) + r_-(z))/2$, where $r_+$
has its poles in $\Cm$ and $r_-$
has its poles in $\Cp$.  By the Schwarz reflection principle,
$r(\overline z ) = \overline{r(z)}$ for all $z\in \C$,
and thus the poles of $r_-$ must be the conjugates of
the poles of $r_+$.  Symmetry further implies
\begin{equation}
r_-(z) = \overline{r_+(\overline{z})} ~~\forall z\in\C, ~\quad
r(x) = \Rrp(x) ~~\forall x\in\R,
\end{equation}
assuming that the constant $r(\infty)$, if it is nonzero, is
split equally between $r_-$ and $r_+$.  Thus
$\Re r_+(z)$ is a bounded harmonic function in $\Cp$
with $\|\Rrp -h\|_\R\le \varepsilon$,
hence also $\|\Rrp-u\|_\Cp \le \varepsilon$ 
by the maximum principle.  Moreover, the poles of $r_+$
are exactly the poles of $r$ in $\Cm$.
Conversely, given $r_+$ as indicated in the second assertion, the function
$r(z) = (r_+(z) + \overline{r_+(\overline z)})/2$ has the required properties.
\end{proof}

The other half of conclusion (1) concerns the case of the open
unit disk $\D$.  Let $S$ denote the unit circle and $\Dm$ the
complement of $\overline{\D}$ in $\C\cup\{\infty\}$.  We get
essentially the same theorem as before.

\begin{thm}
\label{thm2}
Given a real continuous function $h$ on $S$, let $u$ be the
harmonic function in $\D$ with $u(x) =h(x)$ for $x\in S$.  Suppose
there exists a rational
function $r$, also real on $S$, such that $\|r-h\|_S\le \varepsilon$ for
some $\varepsilon \ge 0$.
Then there exists a rational function $r_+$ whose poles are
precisely the poles of\/ $r$ in $\Dm$ such that
$\|\Rrp - h\|_S\le \varepsilon$ and thus also
$\|\Rrp - u\|_\D\le \varepsilon$.  Conversely, if\/ $r_+$ is a rational function
analytic in $\D$ such that $\|\Rrp-u\|_\D\le \varepsilon$, then 
there exists a rational function $r$ whose poles are
the poles of\/ $r_+$ and their reflections in $\D$ such that
$\|r-h\|_S\le \varepsilon$.  
\end{thm}

\begin{proof}
One can argue as before, or alternatively, derive this is a corollary of
Theorem~\ref{thm1} by a M\"obius transformation.
\end{proof}

We now turn to conclusion (2), concerning the case where
$\Omega$ has corners.  As mentioned, in the worst case rational
approximation may give ``too many poles,'' meaning poles not needed
for approximation of the solution of the Laplace problem, and it
may give ``not enough poles,'' meaning poles that are inadequate
to approximate the solution of the Laplace problem.  To explain
this, we present a pair of examples in Figure~\ref{lensfig},
both showing poles of AAA approximations with tolerance $10^{-8}$
on the boundary of the bounded symmetric ``lens'' domain $\Omega$
bounded by two circular arcs meeting at right angles at $z=\pm 1$.

The first image illustrates ``too many poles.''  When the
function $h(z) = \Re z$ is approximated by a rational function
on $\pO$, many poles appear both inside and outside $\Omega$;
this will be the rule almost always when a region has corners.
And yet this boundary data can be exactly matched by the harmonic
function $u(z) = \Re z$, which has just a single pole at $\infty$.
So the clusters of poles obtained by AAA are unnecessary for the
Laplace problem in the interior of $\Omega$.

\begin{figure}[t]
\vskip 15pt
\begin{center}
~~~~~\includegraphics[scale=.93]{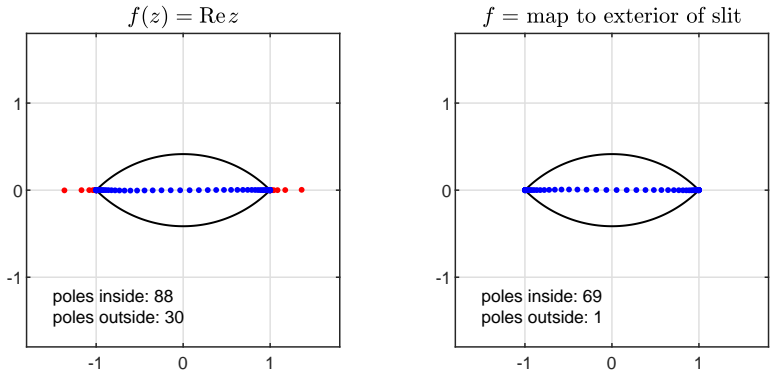}
\caption{\small\label{lensfig}Examples showing that in the worst case,
the pole symmetry principle underlying
the global AAA-LS method may fail.  On the left, AAA approximation gives
``too many poles,'' with poles exponentially clustered outside $\Omega$ near
$\pm 1$ even though the singularity-free function $u(z) = \Re z$ solves the Laplace problem.
On the right, it gives ``too few poles,'' providing no poles
at all outside $\Omega$ near the boundary even though the rational approximation of the
solution of the Laplace problem will need them to approximate the branch
point singularities at $\pm 1$.  Both these situations are
nongeneric and unlikely to appear in practice.}
\end{center}
\end{figure}

The second image illustrates ``too few poles.''  Here $h$ is
taken as the values on $\pO$ of the
analytic function $f$ that maps the exterior of $\Omega$
conformally to the exterior of the slit $[-1,1]$ while leaving
the points $\pm 1$ and $\infty$ fixed:
\begin{equation}
f(z) = {1+v^2\over 1-v^2}, \quad v = -\left( {z-1\over z+1}\right)^{2/3}.
\end{equation}
With the standard branch of the $2/3$ power, $f$ has a branch
cut along $[-1,1]$, and AAA finds a rational approximation $r$
whose poles lie approximately on this slit.  In particular,
they all lie within $\Omega$ apart from one pole of magnitude
$10^{10}$, approximating the pair $f(\infty) = \infty$.  Thus there
are no poles near $\pO$ for the AAA-LS method to work with in
approximating the solution in the interior of $\Omega$, yet this
solution has singularities at $\pm 1$ involving fractional powers
$(z\pm 1)^{4/3}$, so it would need such poles to get high accuracy.

Thus we see that on domains with corners, failure of the pole
symmetry principle is possible.  However, the failures we
have identified are atypical, at least in these extreme forms.
The example on the left in Figure~\ref{lensfig} is special in that
despite the corners in the domain, the solution to the Laplace
problem has no singularities thanks to special boundary data.
This is hardly the generic situation (though picking such examples
is a common mistake beginners make when testing their Laplace
codes!).  As for the example on the right, it has the unusual
property of involving data $h$ that can be analytically continued
to all of $\C\cup\{\infty\}\backslash \overline{\Omega}$.  This is
another very special situation.  Generically, a function $h$ on a
domain boundary with corners will only be analytically continuable
with branch cuts on both sides, and rational approximations will
need to have poles approximating those branch cuts on both sides
of the domain.  Configurations like that of the second image of
Figure~\ref{lensfig} are unlikely to appear in applications.

Now we turn to conclusion (3).  Suppose $\Omega$ is a
simply-connected domain bounded by an analytic curve that is not
simply a circle or straight line.  For such a problem, Schwarz
reflection no longer gives a symmetry equivalence between $\Omega$
and $\C\cup \{\infty\}\backslash \overline{\Omega}$.  What happens
to the pole symmetry principle?

The ``pure mathematics answer'' is that everything works
essentially as before, modified only by the need for a fast
exponentially-convergent polynomial term to be added into
the rational approximations.  The reasoning here can be
based on the technique of considering a conformal map $w =
\phi(z)$ of $\C\cup\{\infty\}\backslash \overline{\Omega}$ to
$\C\cup\{\infty\}\backslash \overline{\Delta}$ with $\phi(\infty) =
\infty$ and its inverse map $z = \psi(w)$~\cite{gaier}.  If $\pO$
is analytic, then $\phi$ and $\psi$ extend analytically to larger
domains, implying that they can be approximated by polynomials in
$z^{-1}$ and $w^{-1}$, respectively, with exponential convergence.
It follows that rational approximation of a function $h$ defined on
$\pO$, for example, is equivalent to rational approximation of its
transplant $\tilde h(w) = h(\psi(w))$ on $S$, up to exponentially
convergent polynomial terms.  If $h$ has singularities, then
root-exponential convergence of rational approximations in $z$
is ensured by the same property for rational approximation of
$\tilde h$ in $w$.  By this kind of reasoning one can argue
that AAA-LS in a smooth domain is like AAA-LS in a disk, up to
constants associated with polynomial approximations.

The ``applied mathematics answer'' is not so simple.  All across
complex analysis, the constants that appear in estimates of
interest tend to grow exponentially as functions of geometric
parameters such as the aspect ratios of reentrant or salient
fingers in boundary curves, and this applies here.  So the
practical status of the pole symmetry principle for regions with
curved boundaries may not be so different from that for regions
with corners.

All the discussion above pertains to the global variant of AAA-LS.\
\ For local variants, as illustrated in the discussion around
the multiply-connected domain of Figure~\ref{circleL}, failures
of the algorithm are more likely to appear in practice if the
AAA step of the algorithm is applied with the data $h$ given.
In such cases, we recommend the method used in that figure:
replace the actual boundary data $h$ by a function $\hat h$
targeted to generate singularities at each corner, such as the
product of the square roots of the distances to the corners.
Our experience shows that as a practical matter, this strategy
is highly effective.  The reason for this is that, though
not all singularities look alike, a wide range of them can be
approximated with root-exponential convergence by exponentially
clustered poles, whose configurations need not be tuned to the
singularities~\cite{lightning,clustering}.  So the set of poles
utilized by AAA to approximate one function will generally also
do well for another.

In the case of a multiply connected domain, to turn to point
(4) of our summary, one should always use a local variant of
the AAA-LS method.  The reason is that approximating harmonic
functions in such a domain will require logarithmic terms since
their conjugates are in general multi-valued~\cite{axler}.  One can
use AAA to approximate a real function $h$ on the boundary $\pO$
of such a domain by a rational function $r$, but $r$ will not have
the right properties interior to $\Omega$.  As illustrated in
Figure~\ref{twoholes}, typically it will approximate different
analytic functions near the different boundary components,
separated by strings of poles approximating branch cuts (compare
Fig.~6.9 of~\cite{AAA}).  These poles have nothing to do with the
harmonic function $u$ in $\Omega$ one wants to approximate, so
in such a case global rational approximations should not be used.

\begin{figure}[h]
\begin{center}
\vspace{10pt}
\includegraphics[scale=.83]{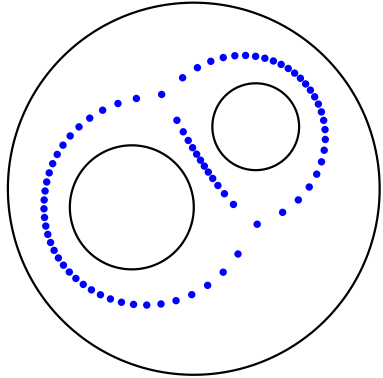}
\caption{\small\label{twoholes}Poles of a global AAA rational approximant $r$ with tolerance
$10^{-8}$ on the boundary of a triply-connected domain with boundary data $0$,
$1$, and $2$ on the smaller, larger, and outer circles, respectively.   The function
$r$ matches the data accurately on all three
parts of $\pO$, but achieves this only by introducing strings of poles
that effectively split $\Omega$ into subdomains with separate analytic functions.  Here,
these are the constant functions $0$, $1$, and $2$, though the configuration would be much
the same for any analytic boundary data.
Effective approximation by a single
harmonic function throughout $\Omega$ would require an additional logarithmic term in each hole,
so for Laplace problems in
domains like this, a local rather than global variant of AAA-LS should be used.}
\end{center}
\end{figure}

In discussing local rational approximations above, we alluded
to a kind of approximate university of pole distributions for
resolving singularities.  This suggests that in the end, AAA
approximation should not really be necessary; one could equally
well use a ``lightning'' strategy in which poles are positioned
a priori rather than determined from the data.  Indeed we think this is
likely to be the case for problems dominated by singular corners,
though the great convenience of starting from AAA approximations
remains an advantage.  For problems less controlled by corners,
global or partially-global variants of AAA-LS will have a power
not easily matched by lightning solvers.

\section{Discussion}
AAA-LS offers a remarkably fast and accurate way to solve Laplace
problems in planar domains with corners.  Typical examples give
8-digit accuracy in a fraction of a second, and the resulting
representation of the solution as the real part of a rational
function can be evaluated in microseconds per point.  Not just the
harmonic function but also its harmonic conjugate are obtained,
thereby giving the analytic extension of the solution in the
problem domain as well as the solution itself---the Hilbert
transform or Dirichlet-to-Neumann map.  For domains with holes,
this analytic extension is a multivalued analytic function, which
consists of a single-valued function plus multivalued log terms,
one for each hole~\cite{axler}.

A feature of all these expansion-based methods is that the
representations of the solution they compute are numerically
nonunique and, a fortiori, non-optimal.  The matrices involved
have enormous or infinite condition numbers, and the coefficient
vectors they deliver may depend in unpredictable ways on details
of boundary discretization and other parameters.  If we solve a
Laplace problem and obtain 8-digit accuracy with 112 poles, for
example, it must not be supposed that these poles are in truly
optimal locations or that 112 is the precise minimal number
for this accuracy.  Despite that, the 8 digits are solid, as
can be verified a posteriori by applying the maximum principle
on a finer boundary grid, and they are achieved thanks to the
regularizing effects of least-squares solvers as realized in the
MATLAB backslash command.

Some other methods for computing rational approximations, such as
vector fitting~\cite{vf}, IRKA~\cite{irka}, RKFIT~\cite{rkfit},
IRF~\cite{hyl}, AGH~\cite{agh}, and the Haut--Beylkin--Monz\'on
reduction algorithm~\cite{hbm}, have optimality as a more central
part of their design concept than AAA-LS, though they too will
often terminate before optimality is achieved.  As a rule, one
can not count on achieving optimality in rational approximation
problems, in view of their extreme sensitivities, which are
reflected both theoretically and computationally in longstanding
complications of spurious poles or ``Froissart doublets''.
For example, it is well known that Pad\'e approximants, which
are defined by optimality in approximating a function and its
derivatives at a single point, do not in general converge to the
function being approximated~\cite{bgm,pade}.

Continuing on the matter of optimality in rational approximation,
we offer an analogy from the field of matrix iterations for
large linear systems of equations $Ax=b$, the core problem of
computational science.  (Actually it is more than an analogy,
since matrix iterations are closely connected with rational
approximations.)  In theory, one might seek to generate
an approximation to the solution vector $x$ at each step of
iteration that was truly optimal by some criterion.  In a sense
this is what certain forms of pure Lanczos or biconjugate gradient
iterations do.  However, it is well known that such an attempt
brings risks of breakdowns and near-breakdowns that interfere
with performance~\cite{fgn}.  In practice, iterative methods
aim for speed rather than optimality, and the idea of trying to
solve $Ax=b$ to a certain accuracy in exactly the minimal number
of steps is not part of the discussion.

In the past few years about a dozen papers have appeared related
to AAA and lightning solution of Laplace problems via rational
approximation and its variants; an impressive example we have
not mentioned is~\cite{baddoo}, and an important earlier work
is~\cite{hyl}.  Most of the methods proposed in these works
approximate continuous boundaries by discrete sets, typically with
thousands of clustered points, and it is an interesting question
to what extent such discretization is necessary.  Even if the
least-squares problem ultimately solved will involve a matrix
with discrete rows, one may wonder whether the discretization
can be deferred or hidden away in ``continuous-mode'' AAA or
AAA-LS methods, as is done by the \hbox{MATLAB} code {\tt
laplace}~\cite{lightningcode} and in Chebfun codes
such as {\tt minimax}.\ \ This is one of many areas in which
AAA and lightning methods, which are very young, can be expected
to improve with further investigation in the years ahead.  We are
also exploring speedups to the linear algebra, and the possibility
of ``log-lightning'' AAA-LS approximation as in~\cite{loglight}.

\section*{Appendix: sample code}
As templates for further explorations, Figures~\ref{code}
and~\ref{codeVA} list the MATLAB codes used to generate the second
row of Figure~\ref{fig1}.

\begin{figure}[h]
{\scriptsize   
\begin{verbatim}
     %% Set up
     s = tanh(linspace(-12,12,300)');
     Z = [1+s; 2+.5i+.5i*s; 1.5+1i+.5*s; 1+1.5i+.5i*s; .5+2i+.5*s; 1i+1i*s];
     w = [0 2 2+1i 1+1i 1+2i 2i].';
     h = @(z) real(z).^2; H = h(Z);
     LW = 'linewidth'; MS = 'markersize'; ms = 6; PO = 'position'; FS = 'fontsize';
     
     %% Local AAA fits
     axes(PO,[.02 .6 .35 .35])
     inpolygonc = @(z,w) inpolygon(real(z),imag(z),real(w),imag(w)); 
     tol = 1e-8; pol_in = []; pol_out = [];
     for k = 1:6
        ii = find(abs(Z-w(k)) == min(abs(Z-w.'),[],2));
        [~,polk] = aaa(H(ii),Z(ii),'tol',tol,'cleanup',0);
        polk_in = polk(inpolygonc(polk,w)); pol_in = [pol_in; polk_in];
        polk_out = polk(~inpolygonc(polk,w)); pol_out = [pol_out; polk_out];
     end
     plot(w([1:end 1]),'k',LW,.9), axis([-.8 2.8 -.8 2.8]), axis square, hold on
     plot(pol_out,'.r',MS,ms), plot(pol_in,'.b',MS,ms), hold off, set(gca,'ytick',0:2)
     title('local AAA poles'), set(gca,FS,6)
     
     %% Solution
     pol = pol_out.';
     d = min(abs(w-pol),[],1);
     [Hes,P] = VAorthog(Z,20); Q = d./(Z-pol); 
     A = [real(P) real(Q) -imag(P) -imag(Q)];
     c = reshape(A\H,[],2)*[1; 1i];
     F = [P Q]*c; U = real(F);
     f = @(z) reshape([VAeval(z(:),Hes) d./(z(:)-pol)]*c,size(z));
     u = @(z) real(f(z));
     
     %% Contour and error plots
     axes(PO,[.35 .6 .35 .35])
     plot(pol,'.r',MS,ms), hold on
     x = linspace(0,2,150); [xx,yy] = meshgrid(x,x); zz = xx+1i*yy;
     uu = u(zz); uu(~inpolygonc(zz,w)) = NaN;
     plot(w([1:end 1]),'k',LW,.9), axis([-.8 2.8 -.8 2.8]), axis square
     contour(x,x,uu,20,LW,1), hold off, set(gca,'ytick',0:2)
     u99err = u(.99+.99i) - 1.0267919261073
     title('Laplace solution'), set(gca,FS,6)
     axes(PO,[.73 .6 .25 .35])
     semilogy(angle(Z-(.5+.5i)),abs(U-H),'.k',MS,3), grid on
     set(gca,FS,6), axis([-pi pi 1e-12 1e-4])
     set(gca,'xtick',pi*(-1:.5:1),'xticklabel',{'-\pi','-\pi/2','0','\pi/2','\pi'})
     title('error against angle'), set(gca,FS,6)
\end{verbatim}
\par\vspace{-10pt}}
\caption{\small\label{code}MATLAB code to generate the second row of Figure~\ref{fig1}.}
\end{figure}

\begin{figure}[t]
{\scriptsize   
\begin{verbatim}
     function [Hes,R] = VAorthog(Z,n,varargin)  % Vand.+Arnoldi orthogonalization
     %  Input:   Z = column vector of sample points
     %           n = degree of polynomial (>= 0)
     %         Pol = cell array of vectors of poles (optional)
     % Output: Hes = cell array of Hessenberg matrices (length 1+length(Pol))
     %           R = matrix of basis vectors
     M = length(Z); Pol = []; if nargin == 3, Pol = varargin{1}; end
     % First orthogonalize the polynomial part
     Q = ones(M,1); H = zeros(n+1,n);
     for k = 1:n       
        q = Z.*Q(:,k);
        for j = 1:k, H(j,k) = Q(:,j)'*q/M; q = q - H(j,k)*Q(:,j); end 
        H(k+1,k) = norm(q)/sqrt(M); Q(:,k+1) = q/H(k+1,k);
     end
     Hes{1} = H; R = Q;
     % Next orthogonalize the pole parts, if any
     while ~isempty(Pol)
        pol = Pol{1}; Pol(1) = [];
        np = length(pol); H = zeros(np,np-1); Q = ones(M,1);
        for k = 1:np       
           q = Q(:,k)./(Z-pol(k));
           for j = 1:k, H(j,k) = Q(:,j)'*q/M; q = q - H(j,k)*Q(:,j); end 
           H(k+1,k) = norm(q)/sqrt(M); Q(:,k+1) = q/H(k+1,k);
        end
        Hes{length(Hes)+1} = H; R = [R Q(:,2:end)];
     end
\end{verbatim}
\begin{verbatim}

     function [R0,R1] = VAeval(Z,Hes,varargin)  % Vand.+Arnoldi basis construction
     %  Input:   Z = column vector of sample points
     %         Hes = cell array of Hessenberg matrices
     %         Pol = cell array of vectors of poles, if any
     % Output:  R0 = matrix of basis vectors for functions
     %          R1 = matrix of basis vectors for derivatives
     M = length(Z); Pol = []; if nargin == 3, Pol = varargin{1}; end
     % First construct the polynomial part of the basis
     H = Hes{1}; Hes(1) = []; n = size(H,2); 
     Q = ones(M,1); D = zeros(M,1);
     for k = 1:n
        hkk = H(k+1,k);
        Q(:,k+1) = ( Z.*Q(:,k) - Q(:,1:k)*H(1:k,k)          )/hkk;
        D(:,k+1) = ( Z.*D(:,k) - D(:,1:k)*H(1:k,k) + Q(:,k) )/hkk;
     end
     R0 = Q; R1 = D;
     % Next construct the pole parts of the basis, if any
     while ~isempty(Pol)
        pol = Pol{1}; Pol(1) = [];
        H = Hes{1}; Hes(1) = []; np = length(pol); Q = ones(M,1); D = zeros(M,1);
        for k = 1:np
           Zpki = 1./(Z-pol(k)); hkk = H(k+1,k);
           Q(:,k+1) = ( Q(:,k).*Zpki - Q(:,1:k)*H(1:k,k)                   )/hkk;
           D(:,k+1) = ( D(:,k).*Zpki - D(:,1:k)*H(1:k,k) - Q(:,k).*Zpki.^2 )/hkk;
        end
        R0 = [R0 Q(:,2:end)]; R1 = [R1 D(:,2:end)];
     end
\end{verbatim}
\par}
\caption{\small\label{codeVA}Codes for Vandermonde with Arnoldi orthogonalization and evaluation,
from~\cite{stokes}.}
\end{figure}

\newpage

\begin{acknowledgments}
We are grateful for advice and collaboration over the past
few years from Peter Baddoo,
Pablo Brubeck, Abi Gopal, Yuji Nakatsukasa, Kirill Serkh,
Andr\'e Weideman, and Heather Wilber.
\end{acknowledgments}

\par

{
\small

\par}

\end{document}